\documentclass[12pt]{article}

\usepackage{amssymb}
\usepackage{amsthm}
\usepackage{amsmath}
\usepackage{graphicx}
\usepackage{fullpage}
\usepackage{enumerate}
\usepackage{color}
 \numberwithin{equation}{section}

\theoremstyle{plain}
\newtheorem{thm}{Theorem}[section]
\newtheorem{cor}[thm]{Corollary}

\newtheorem{lem}[thm]{Lemma}
\newtheorem{prop}[thm]{Proposition}

\theoremstyle{definition}
\newtheorem{defn}[thm]{Definition}

\theoremstyle{remark}

\newtheorem{rem}[thm]{Remark}

\newcommand{\N}{\mathbb{N}}
\newcommand{\R}{\mathbb{R}}

\newcommand{\I}{\infty}

\newcommand{\bp}{\begin{proof}[\ensuremath{\mathbf{Proof}}]}
\newcommand{\ep}{\end{proof}}
\newcommand{\e}{\varepsilon}
\newcommand{\lap}{\Delta}

\begin{document}


\title{A doubly nonlinear evolution for the optimal \\ Poincar\'{e} inequality}

\author{Ryan Hynd\footnote{Department of Mathematics, University of Pennsylvania.  Partially supported by NSF grant DMS-1301628.}\; and Erik Lindgren\footnote{Department of Mathematics, KTH. Supported by the Swedish Research Council, grant no. 2012-3124. Partially supported by the Royal Swedish Academy of Sciences.}}  

\maketitle

\begin{abstract}
We study the large time behavior of solutions of the PDE $|v_t|^{p-2}v_t=\Delta_p v$. A special property of this equation is that the Rayleigh quotient $\int_{\Omega}|Dv(x,t)|^pdx /\int_{\Omega}|v(x,t)|^pdx$ is nonincreasing in time along solutions.  As $t$ tends to infinity, this ratio converges to the optimal constant in Poincar\'{e}'s inequality. 
Moreover, appropriately scaled solutions converge to a function for which equality holds in this inequality.  An interesting limiting equation also arises when $p$ tends to infinity, which provides a new approach to approximating ground states of the infinity Laplacian. 
\end{abstract}

\section{Introduction}
In this paper, we study solutions $v:\Omega\times(0,\infty)\rightarrow\R$ of the PDE
\begin{equation}\label{MainPDE}
|v_t|^{p-2}v_t=\Delta_p v
\end{equation}
where $\Omega\subset \R^d$ is a bounded domain, $p\in (1,\infty)$, and $\Delta_p$ is the $p$-Laplacian
$$
\Delta_p \psi:=\text{div}(|D\psi|^{p-2}D\psi).
$$
The $p$-Laplacian arises in connection with various physical applications. Examples include non-Newtonian fluids, nonlinear
elasticity, glacial sliding and capillary surfaces as detailed in  \cite{Appell}, \cite{AJ92}, \cite{Krist} and \cite{Pel75}.

\par Observe that when $p=2$, the PDE \eqref{MainPDE} is the heat equation. As a result, we view \eqref{MainPDE} as a nonlinear flow.  What separates 
equation \eqref{MainPDE} from typical nonlinear parabolic equations, is the nonlinearity in the time derivative $|v_t|^{p-2}v_t$.  This type of equation is known in the literature 
as a {\it doubly nonlinear evolution}.  Furthermore, we regard \eqref{MainPDE} as special within the class of doubly nonlinear evolutions as it is homogeneous: if $v$ is a solution of  \eqref{MainPDE},  any multiple of $v$ is also a solution.

\par Our motivation for studying equation \eqref{MainPDE} is its connection with the optimal Poincar\'{e} inequality 
\begin{equation}\label{PoincareIneq}
\lambda_p\int_{\Omega}|\psi|^pdx\le \int_{\Omega}|D\psi|^pdx, \quad\psi\in W^{1,p}_0(\Omega).
\end{equation}
Here
$$
\lambda_p:=\inf\left\{\frac{\int_{\Omega}|D\psi|^pdx }{\int_{\Omega}|\psi|^pdx }: \psi\in W^{1,p}_0(\Omega)\setminus\{0\}\right\}
$$
is the least $p$-Rayleigh quotient, and \eqref{PoincareIneq} is ``optimal" as $\lambda_p$ is the largest constant for which this inequality is valid.  A function $\psi\in W^{1,p}_0(\Omega)\setminus\{0\}$ for which equality holds in \eqref{PoincareIneq} is called a {\it ground 
state} of $p$-Laplacian or simply a {\it $p$-ground state}. These functions are easily seen to exist and to satisfy the PDE
\begin{equation}\label{GroundStatePDE}
-\Delta_p \psi=\lambda_p |\psi|^{p-2}\psi
\end{equation}
in $\Omega$.  Moreover, $\lambda_p$ is ``simple'' in the sense that any two $p$-ground states are multiples of each other \cite{LindKaw, Lin,  Saka}. 

\par In what follows, we prove that a properly scaled solution of the initial value problem
\begin{equation}\label{pParabolic}
\begin{cases}
|v_t|^{p-2}v_t=\Delta_p v, \quad \Omega\times (0,\infty) \\
\hspace{.48in} v=0, \hspace{.31in} \partial \Omega\times [0,\infty)\\
\hspace{.48in} v=g, \hspace{.32in} \Omega\times\{0\}
\end{cases}
\end{equation}
converges to a $p$-ground state as $t\rightarrow\infty$. First, we show that \eqref{pParabolic} has a {\it weak solution} in the sense of a doubly nonlinear evolution, and then derive various 
global estimates on weak solutions.  In particular, we verify that the $p$-Rayleigh quotient is nonincreasing for each weak solution of 
\eqref{pParabolic}
$$
\frac{d}{dt}\left\{\frac{\int_{\Omega}|Dv(x,t)|^pdx }{\int_{\Omega}|v(x,t)|^pdx }\right\}\le 0.
$$
This monotonicity formula and the homogeneity of equation \eqref{MainPDE} are crucial ingredients in establishing the following result.

\begin{thm}\label{LargeTthm} Assume $g\in W^{1,p}_0(\Omega)$ and define   
\begin{equation}\label{mup}
\mu_p:=\lambda_p^{\frac{1}{p-1}}.\nonumber
\end{equation}
Then for any weak solution $v$ of \eqref{pParabolic}, the limit 
$$
\psi:=\lim_{t\rightarrow \infty}e^{\mu_p t}v(\cdot,t)
$$ exists in $W^{1,p}_0(\Omega)$ and is a $p$-ground state, provided $\psi\not\equiv 0$. In this case, $v(\cdot, t)\not\equiv 0$ for $t\ge 0$ and 
$$
\lambda_p=\lim_{t\rightarrow \infty}\frac{\int_\Omega|Dv(x,t)|^pdx}{\int_\Omega|v(x,t)|^pdx}. 
$$
\end{thm} 
When $p=2$, a direct proof of Theorem \ref{LargeTthm} can be made by writing the solution of the heat equation in terms of the basis of eigenfunctions for the Dirichlet Laplacian.  For $p\neq 2$, no such formulae are available and we must work directly with the equation.  It is interesting to compare Theorem \ref{LargeTthm} to other large time asymptotics results for fully nonlinear parabolic equations \cite{Arm, KimLee} and for 
nonlinear degenerate flows \cite{ABC, Aronson,Vazquez2, Vazquez1}.  Most of these works involve comparison principles and initial conditions which do not change sign.
Our main tool in this paper is a compactness property of weak solutions of \eqref{MainPDE} and applies to general initial data.  

\par We also verify that \eqref{pParabolic} has a unique {\it viscosity solution} when $p\ge 2$. We note it is unknown whether weak solutions are unique or if each weak solution is  a viscosity solution.   
Moreover, the uniqueness of solutions of general doubly nonlinear evolutions is not well understood. Nevertheless, we show there is always one weak solution of \eqref{pParabolic} that arises via the implicit time scheme: 
$v^0=g$, 
\begin{equation}
\begin{cases}\label{ViscScheme}
{\cal J}_p\left(\frac{v^k - v^{k-1}}{\tau}\right)=\Delta_pv^k,& x\in \Omega \\
\hspace{.76in} v^k=0, & x\in\partial \Omega
\end{cases}
\end{equation}
for $k\in\N$ and  $\tau>0$.  Here  ${\cal J}_p$ is the increasing function
\begin{equation}
{\cal J}_p(w):=|w|^{p-2}w, \quad w\in\R. \nonumber
\end{equation}
Standard variational methods can be used to show this scheme has a unique weak solution sequence $\{v^k_\tau\}_{k\in \N}\subset W^{1,p}_0(\Omega)$ for each $\tau>0$.  We argue that each $v^k_\tau$ is also a continuous viscosity solution and then use viscosity solutions methods to verify the following convergence result.

\begin{thm}\label{ImplicitSchemeThm}
Assume that $p\ge 2$ and that $\partial\Omega$ is smooth.  Additionally suppose that $g\in W^{1,p}_0(\Omega)\cap C(\overline{\Omega})$ and that there is a $p$-ground state $\varphi$ such that
$$
-\varphi(x)\le g(x)\le \varphi(x), \quad x\in \overline{\Omega}.
$$
Denote the solution sequence of the implicit scheme \eqref{ViscScheme} as
$\{v^k_\tau\}_{k\in \N}$ and set 
\begin{equation}\label{CLaan}
v_N(\cdot, t):=
\begin{cases}
g, \hspace{.4in} t=0\\
v^k_{T/N}, \quad (k-1)T/N< t \le kT/N, \quad k=1, \dots, N\\
\end{cases}
\end{equation}
for $N\in\N$ and $T>0$.  Then $v(\cdot,t):=\lim_{N\rightarrow \infty}v_N(\cdot,t)$ exists in  $L^p(\Omega)\cap C(\overline{\Omega})$ uniformly in $t\in [0,T]$. Moreover, $v$ is the unique viscosity solution and a weak solution of the initial value problem \eqref{pParabolic}.
\end{thm}

\par It was previously established that a subsequence of $(v_N)_{N\in \N}$ converges to a weak solution \cite{Arai, Colli}.  The novelty of Theorem \ref{ImplicitSchemeThm} is that the full limit exists and that the limit is additionally a viscosity solution. Employing viscosity solutions will also allow us to pass to the limit as the exponent $p\rightarrow \infty$ in equation \eqref{MainPDE}.  This idea was inspired by the work of P. Juutinen, P. Lindqvist and J. Manfredi, who first studied the so-called infinity eigenvalue problem and infinity ground states \cite{JLM}.  We view the following result as providing a natural evolution equation for the infinity eigenvalue problem and its ground states.

\begin{thm}\label{pinfProp}
Assume $g\in W^{1,\infty}_0(\Omega)$ and let $v^p$ denote a viscosity solution of \eqref{pParabolic} for $p\ge 2$ with initial 
condition $g$.  There is an increasing sequence $p_k\rightarrow \infty$ such that $(v^{p_k})_{k\in \N}$ converges locally uniformly to a viscosity solution of the PDE 
\begin{equation}\label{infParabolic}
\begin{cases}
G_\infty(v_t, Dv, D^2v)=0,  \hspace{.31in} \Omega\times (0,\infty)\\
\hspace{1.1in} v=0, \hspace{.31in} \partial \Omega\times [0,\infty)\\
\hspace{1.1in} v=g,\hspace{.32in} \Omega\times\{0\}
\end{cases}
\end{equation}
as $k\rightarrow \infty$. The operator above is defined as  
$$
G_\infty(\phi_t,D\phi, D^2\phi):=
\begin{cases}
\min\{-\Delta_\infty\phi, |D\phi| + \phi_t\}, \quad &\phi_t<0\\
-\Delta_\infty\phi, \quad & \phi_t=0\\
\max\{-\Delta_\infty\phi,- |D\phi| + \phi_t\}, \quad& \phi_t>0\\
\end{cases},
$$
where $\Delta_\infty\phi:=D^2\phi D\phi\cdot D\phi$ is the infinity Laplacian.  
\end{thm}

\par This paper is organized as follows. In section \ref{WeakSol}, we discuss the existence theory for weak solutions.  In particular, we present a novel compactness result for the doubly nonlinear evolution \eqref{pParabolic}. We justify Theorem \ref{LargeTthm} in section \ref{LargeTlim} and then  discuss viscosity solutions and prove Theorem \ref{ImplicitSchemeThm} in section \ref{ViscSoln}.  Finally, we verify Theorem \ref{pinfProp} in section \ref{LargePlim}. 
We thank the Institut Mittag-Leffler for hosting us during the initial phase of this work. We especially thank Peter Lindqvist and Jerry Kazdan for their advice and encouragement.

\section{Weak Solutions}\label{WeakSol}
An important identity for smooth solutions of \eqref{pParabolic} is 
\begin{equation}\label{diffEnergyIdentity}
\frac{d}{dt}\int_{\Omega}\frac{1}{p}|Dv(x,t)|^pdx=-\int_{\Omega}|v_t(x,t)|^pdx.
\end{equation}
This identity follows from direct computation.  Of course, integrating \eqref{diffEnergyIdentity} in time yields
\begin{equation}\label{EnergyIdentity}
\int^t_0\int_{\Omega}|v_t(x,s)|^pdxds+\int_{\Omega}\frac{1}{p}|Dv(x,t)|^pdx=\int_{\Omega}\frac{1}{p}|Dg(x)|^pdx
\end{equation}
for $t\ge 0$.  This resulting equality leads us to seek solutions defined as follows. 
\begin{defn}
Assume $g\in W^{1,p}_0(\Omega)$.  We say that a function $v$ satisfying 
\begin{equation}\label{NaturalSpace}
v\in L^\I([0,\infty); W^{1,p}_0(\Omega)), \quad v_t\in L^p(\Omega\times [0,\infty))
\end{equation}
is a {\it weak solution} of \eqref{pParabolic} if for Lebesgue almost every $t> 0$
\begin{equation}\label{WeakForm}
\int_{\Omega}|v_t(x,t)|^{p-2}v_t(x, t)\phi(x)dx +\int_{\Omega}|Dv(x,t)|^{p-2}Dv(x, t)\cdot D\phi(x)dx =0 \\
\end{equation}
for each $\phi\in W^{1,p}_0(\Omega)$ and 
\begin{equation}\label{InitialCond}
v(x,0)=g(x). 
\end{equation}
\end{defn}
Any $v$ satisfying \eqref{NaturalSpace} takes values in $L^p(\Omega)$ that are continuous in time, that is
$$
v\in C([0,T]; L^p(\Omega))\quad \text{for any $T>0$.}
$$
Therefore, we may consider the pointwise values $v(\cdot,t)\in L^p(\Omega)$ of a weak solution and assign the initial condition \eqref{InitialCond}. Let us now 
derive a few properties of solutions.

\begin{lem}\label{AbsPhi}
Assume $v$ is a weak solution of \eqref{pParabolic}. Then $[0,\infty)\ni t\mapsto \int_\Omega |Dv(x,t)|^pdx$ is absolutely continuous and \eqref{diffEnergyIdentity} holds for almost every $t>0$. 
\end{lem}
\begin{proof} Define  
$$
\Phi(w):=
\begin{cases}
\int_\Omega \frac{1}{p}|Dw(x)|^pdx, \quad &w\in W^{1,p}_0(\Omega)\\
+\infty, &\text{otherwise}
\end{cases}
$$
for each $w\in L^p(\Omega)$. Observe that $\Phi$ is convex, proper, and lower-semicontinuous. Moreover, by \eqref{WeakForm} 
$$
\partial\Phi(v(\cdot,t))=\{-|v_t(\cdot, t)|^{p-2}v_t(\cdot, t)\}
$$
for almost every $t>0$.  In view of the integrability of $v_t$ \eqref{NaturalSpace}, it follows that $t\mapsto \Phi(v(\cdot,t))$ is absolutely continuous; for instance, see Corollary 1.4.5 and Remark 1.4.6 of \cite{AGS} for a detailed proof of this fact.  The chain rule now applies, and \eqref{diffEnergyIdentity} holds for almost every $t>0$.
\end{proof}

\begin{lem}\label{LemmaINeq}
Assume $v$ is a weak solution of \eqref{pParabolic}. Then
\begin{equation}\label{Dvvtbound}
\int_{\Omega}|Dv(x,t)|^pdx\le \frac{1}{\mu_p}\int_{\Omega}|v_t(x,t)|^pdx
\end{equation}
and 
\begin{equation}\label{ExpDecay}
\frac{d}{dt}\left\{e^{(\mu_p p)t}\int_{\Omega}|Dv(x,t)|^pdx\right\}\le 0
\end{equation} 
for almost every $t\ge 0$.

\end{lem}
\begin{proof} Using $v(\cdot, t)$ as a test function in \eqref{WeakForm} and applying Poincar\'{e}'s inequality \eqref{PoincareIneq}
\begin{align}\label{HolderDerBound}
\int_{\Omega}|Dv(x,t)|^pdx&=\int_{\Omega}|Dv(x,t)|^{p-2}Dv(x,t)\cdot Dv(x,t)dx \nonumber \\
&=-\int_{\Omega}|v_t(x,t)|^{p-2}v_t(x,t)\cdot v(x,t) dx \nonumber \\
&\le \left(\int_{\Omega}|v_t(x,t)|^pdx\right)^{1-1/p}\left(\int_{\Omega}|v(x,t)|^pdx\right)^{1/p}  \\
&\le \lambda_p^{-1/p}\left(\int_{\Omega}|v_t(x,t)|^pdx\right)^{1-1/p}\left(\int_{\Omega}|Dv(x,t)|^pdx\right)^{1/p}.\nonumber
\end{align}
This proves \eqref{Dvvtbound}.  Combining \eqref{diffEnergyIdentity} and \eqref{Dvvtbound} gives
\begin{equation}\label{DiffIneqmu}
\frac{d}{dt}\int_{\Omega}|Dv(x,t)|^pdx \le - p \mu_p \int_{\Omega}|Dv(x,t)|^pdx.
\end{equation}
Inequality \eqref{ExpDecay} follows from \eqref{DiffIneqmu} by direct computation. 
\end{proof}
Note that if the initial condition $g$ is a $p$-ground state, then
\begin{equation}\label{pSepVarSoln}
v(x,t)=e^{-\mu_p t}g(x)
\end{equation}
is a solution of \eqref{pParabolic}.  Theorem \ref{LargeTthm} asserts all solutions exhibit this ``separation of variables" behavior in the limit as $t\rightarrow \infty$. Our first clue 
that this intuition is correct is that the $p$-Rayleigh quotient is a nonincreasing function of time along the flow.  We regard this as a special 
feature of the PDE \eqref{MainPDE}.

\begin{prop}\label{RayleighGoDown} Assume that $v$ is a weak solution of \eqref{pParabolic} such that $v(\cdot, t)\neq 0\in L^p(\Omega)$ for each 
$t\ge 0$. Then the $p$-Rayleigh quotient 
$$
[0,\infty)\ni t\mapsto \frac{\int_{\Omega}|Dv(x,t)|^pdx }{\int_{\Omega}|v(x,t)|^pdx }
$$
is nonincreasing.
\end{prop}
\begin{proof}
Employing \eqref{NaturalSpace}, it is not difficult to verify  
$$
\frac{d}{dt}\int_\Omega\frac{1}{p}|v(x,t)|^pdx=\int_\Omega|v(x,t)|^{p-2}v(x,t)v_t(x,t)dx
$$
for almost every time $t>0$; for instance, it is possible to adapt the proof of Theorem 3 on page 287 of \cite{LCE}. Suppressing the $(x,t)$ dependence, we compute using \eqref{diffEnergyIdentity}
\begin{align}\label{RayleighComp}
\frac{d}{dt}\frac{\int_{\Omega}|Dv|^pdx }{\int_{\Omega}|v|^pdx } & = 
-p\frac{\int_{\Omega}|v_t|^pdx }{\int_{\Omega}|v|^pdx } - p \frac{\int_{\Omega}|Dv|^pdx }{\left(\int_{\Omega}|v|^pdx\right)^2} \int_{\Omega}|v|^{p-2} vv_tdx \nonumber \\
&=\frac{p}{\left(\int_{\Omega}|v|^pdx\right)^2}\left\{\int_{\Omega}|Dv|^pdx\int_{\Omega}|v|^{p-2} v(-v_t)dx -\int_{\Omega}|v|^pdx \int_{\Omega}|v_t|^pdx \right\}
\end{align}
which is valid for almost every $t>0$.  By H\"{o}lder's inequality 
$$
\int_{\Omega}|v|^{p-2} v(-v_t)dx\le \left(\int_{\Omega}|v|^{p}dx\right)^{1-1/p} \left(\int_{\Omega}|v_t|^{p}dx\right)^{1/p},
$$
and combining this with \eqref{HolderDerBound} gives
$$
\int_{\Omega}|Dv|^pdx\int_{\Omega}|v|^{p-2} v(-v_t)dx \le \int_{\Omega}|v|^pdx \int_{\Omega}|v_t|^pdx.
$$ 
From \eqref{RayleighComp}, we conclude
$$
\frac{d}{dt} \frac{\int_{\Omega}|Dv|^pdx }{\int_{\Omega}|v|^pdx }\le 0.
$$
\end{proof}
\begin{cor}\label{UniqueGround}
Assume $g$ is a $p$-ground state. The only weak solution of initial value problem \eqref{pParabolic} is given by \eqref{pSepVarSoln}.
\end{cor}
\begin{proof}
Let $v$ be a weak solution of \eqref{pParabolic} and assume initially that $v(\cdot, t)\neq 0\in L^{p}(\Omega)$ for each $t\ge 0$.  By Proposition \ref{RayleighGoDown}, 
$$
\frac{\int_{\Omega}|Dv(x,t)|^pdx }{\int_{\Omega}|v(x,t)|^pdx }\le \frac{\int_{\Omega}|Dg(x)|^pdx }{\int_{\Omega}|g(x)|^pdx }=\lambda_p.
$$
Thus, $v(\cdot, t)$ is a $p$-ground state for each $t\ge 0$.  In view of equation \eqref{GroundStatePDE} 
$$
|v_t|^{p-2}v_t=\Delta_p v= - \lambda_p |v|^{p-2}v.
$$
In particular, 
\begin{equation}\label{ODEvee}
v_t=-\mu_p v
\end{equation}
and therefore, $v$ is given by \eqref{pSepVarSoln}.

\par Otherwise, select the first time $T>0$ for which $v(\cdot, T)=0\in L^{p}(\Omega)$.  By our argument above, $v(\cdot, t)$ is a $p$-ground state for each $t\in [0,T)$. Moreover, 
\eqref{ODEvee} holds for almost every $t\in (0,T)$.  However, this implies $v(\cdot,T)=e^{-\mu_p T}g\neq 0\in L^{p}(\Omega)$. Therefore, there is no such time $T$ and $v$ 
is given by \eqref{pSepVarSoln}. 
\end{proof}

\par Using an implicit time scheme such as \eqref{ViscScheme} to solve doubly nonlinear evolutions in 
reflexive Banach spaces has been carried out with great success; see \cite{AGS, Arai, Colli, Colli2,LCE2, Savare,Stef}. In our view, the main insight that makes this approach work is a certain compactness feature 
of weak solutions that we now explore. Roughly, we verify that any ``bounded" sequence of solutions has a subsequence converging to another weak solution.  We will also 
make use of this compactness result in our study of the large time behavior of solutions. 

\begin{thm}\label{CompactnessLem}
Assume $\{g^k\}_{k\in \N}\in W_0^{1,p}(\Omega)$ is uniformly bounded in $W_0^{1,p}(\Omega)$, and that for each $k\in \N$, $v^k$ is a weak solution of \eqref{pParabolic} with $v^k(\cdot,0)=g^k$. Then there is a subsequence $\{v^{k_j}\}_{j\in \N}$ and $v$ satisfying 
\eqref{NaturalSpace} such that 
\begin{equation}\label{FirstConv}
v^{k_j}\rightarrow v\quad \text{in}\quad
\begin{cases}
C([0,T]; L^p(\Omega)) \\
L^p([0,T]; W^{1,p}_0(\Omega))
\end{cases}
\end{equation}
and
\begin{equation}\label{SecondConv}
v_t^{k_j}\rightarrow v_t\quad \text{in}\quad L^p(\Omega\times [0,T])
\end{equation}
as $j\rightarrow \infty$, for all $T>0$.  Moreover, $v$ is a weak solution of \eqref{pParabolic} where $g$ is a weak limit of $\{g^{k_j}\}_{k\in \N}$ in $W_0^{1,p}(\Omega)$. 
\end{thm}
\begin{proof}
By equation \eqref{diffEnergyIdentity}, we have for each $k\in \N$ and almost every time $t\ge 0$
\begin{equation}\label{Energyk}
\frac{d}{dt}\int_{\Omega}\frac{|Dv^k(x,t)|^p}{p}dx=-\int_{\Omega}|v^k_t(x,t)|^pdx.
\end{equation}
Thus,
\begin{equation}\label{EnergyIdk}
\int^\infty_0\int_\Omega |v^k_t(x,t)|^pdxdt +\sup_{t\ge 0}\int_\Omega |Dv^k(x,t)|^pdx \le 2\int_\Omega |Dg^k(x)|^pdx.
\end{equation}
By assumption, the right hand side above is bounded uniformly in $k\in \N$.  By the compactness of $W^{1,p}_0(\Omega)$ in $L^p(\Omega)$, 
the Arzel\`{a}-Ascoli theorem as detailed by J. Simon \cite{Simon} implies that there is a subsequence $\{v^{k_j}\}_{j\in \N}$ converging 
uniformly on compact subintervals of $[0,\infty)$ to some $v$ in $L^p(\Omega)$.  

\par The bound \eqref{EnergyIdk} also ensures
$$
Dv^{k_j}(\cdot, t)\rightharpoonup Dv(\cdot, t)
$$ in $L^p(\Omega; \R^n)$ for each $t\ge 0$.  Moreover, as $\{v_t^{k}\}_{k\in \N}$ is bounded in $L^p(\Omega\times [0,\infty))$,  we may also assume 
$$
\begin{cases}
v_t^{k_j}\rightharpoonup v_t\quad \text{in}\quad L^p(\Omega\times [0,\infty))\\
{\cal J}_p(v_t^{k_j})\rightharpoonup \xi\quad \text{in}\quad L^q(\Omega\times [0,\infty))
\end{cases}.
$$
Here $1/p+1/q=1$. We claim that in fact 
\begin{equation}\label{needWeak}
\xi={\cal J}_p(v_t)=|v_t|^{p-2}v_t.
\end{equation}
\par The convexity of the map $\R^n\ni z\mapsto \frac{1}{p}|z|^{p}$ implies 
$$
\int_\Omega \frac{1}{p}|Dw(x)|^pdx\ge  \int_\Omega \frac{1}{p}|Dv^{k_j}(x,t)|^pdx -\int_{\Omega}{\cal J}_p(v_t^{k_j}(x,t))(w(x)-v^{k_j}(x,t))dx
$$
for any $w\in W^{1,p}_0(\Omega)$.   Integrating over the interval $t\in [t_0, t_1]$ and sending $j\rightarrow \infty$ gives 
$$
\int^{t_1}_{t_0}\int_\Omega \frac{1}{p}|Dw(x)|^pdxdt \ge  \int^{t_1}_{t_0}\left(\int_\Omega \frac{1}{p}|Dv(x,t)|^pdx -\int_{\Omega}\xi(x,t)(w(x)-v(x,t))dx\right)dt.
$$
Therefore, 
$$
\int_\Omega \frac{1}{p}|Dw(x)|^pdx\ge\int_\Omega\frac{1}{p}|Dv(x,t)|^pdx -\int_{\Omega}\xi(x,t)(w(x)-v(x,t))dx
$$
for almost every time $t\ge 0$. In particular, for each $\phi \in W^{1,p}_0(\Omega)$
\begin{equation}\label{haveWeak}
\int_{\Omega}\xi(x, t)\phi(x)dx +\int_{\Omega}|Dv(x,t)|^{p-2}Dv(x, t)\cdot D\phi(x)dx =0 
\end{equation}
for almost every time $t\ge 0$.  As a result, once we verify \eqref{needWeak}, $v$ is then a weak solution of \eqref{pParabolic}.  

\par For each interval $[t_0, t_1]$
\begin{align*}
\lim_{j\rightarrow \infty}\int^{t_1}_{t_0}\int_\Omega|Dv^{k_j}(x,t)|^pdxdt & =\lim_{j\rightarrow \infty}\int^{t_1}_{t_0}\int_\Omega|Dv^{k_j}(x,t)|^{p-2}Dv^{k_j}(x,t)\cdot Dv^{k_j}(x,t)dxdt\\
&=-\lim_{j\rightarrow \infty}\int^{t_1}_{t_0}\int_\Omega {\cal J}_p(v_t^{k_j}(x,t)) v^{k_j}(x,t)dxdt\\
&=-\int^{t_1}_{t_0}\int_\Omega \xi(x,t) v(x,t)dxdt\\
&=\int^{t_1}_{t_0}\int_\Omega|D v(x,t)|^pdxdt.
\end{align*}
The last equality is due to \eqref{haveWeak}.  As a result, $Dv^{k_j}\rightarrow Dv$ in $L^p_\text{loc}([0,\infty), L^{p}(\Omega))$. This proves assertion \eqref{FirstConv}.  And without loss of generality, we assume that
\begin{equation}\label{StrongConverge}
\int_\Omega|Dv^{k_j}(x,t)|^pdx\rightarrow \int_\Omega|Dv(x,t)|^pdx
\end{equation}
for almost every $t\ge 0$, as $j\rightarrow \infty$ (since
this occurs for some subsequence of $k_j$). 

\par Now we will verify \eqref{needWeak}. As in our proof of Lemma \ref{AbsPhi}, \eqref{haveWeak} implies 
$$
\frac{d}{dt}\int_\Omega \frac{1}{p}|Dv(x,t)|^pdx= -\int_{\Omega}\xi(x,t)v_t(x,t)dx, \quad a.e. \; t\ge 0. 
$$
Thus for each $t_1> t_0$
\begin{equation}\label{NeedtoCompare}
\int^{t_1}_{t_0}\int_{\Omega}\xi(x,s)v_t(x,s)dxds+\int_{\Omega}\frac{|Dv(x,t_1)|^p}{p}dx=\int_{\Omega}\frac{|Dv(x,t_0)|^p}{p}dx.
\end{equation}
From \eqref{Energyk}, we may also write
\begin{equation}\label{NeedtoCompare2}
\int^{t_1}_{t_0}\int_{\Omega}\frac{1}{p}|v^{k_j}_t(x,s)|^p+\frac{1}{q}|{\cal J}_p(v_t^{k_j}(x,s))|^qdxds+\int_{\Omega}\frac{|Dv^{k_j}(x,t_1)|^p}{p}dx= \int_{\Omega}\frac{|Dv^{k_j}(x,t_0)|^p}{p}dx.
\end{equation}
\par Assuming $t_0$ and $t_1$ are times for which the limit \eqref{StrongConverge} holds, we let $j\rightarrow \infty$ to get
$$
\int^{t_1}_{t_0}\int_{\Omega}\frac{1}{p}|v_t(x,s)|^p+\frac{1}{q}|\xi(x,s)|^qdxds+\int_{\Omega}\frac{|Dv(x,t_1)|^p}{p}dx\le  \int_{\Omega}\frac{|Dv(x,t_0)|^p}{p}dx
$$
by weak convergence.  Comparing with \eqref{NeedtoCompare} gives 
$$
\int^{t_1}_{t_0}\int_{\Omega}\left(\frac{1}{p}|v_t(x,s)|^p+\frac{1}{q}|\xi(x,s)|^q -  \xi(x,s)v_t(x,s)\right)dxds\le 0.
$$
Equation \eqref{needWeak} now follows from the strict convexity of $\R\ni z\mapsto \frac{1}{p}|z|^p$.   Substituting $\xi={\cal J}_p(v_t)$ into \eqref{NeedtoCompare} and passing to the limit as $j\rightarrow \infty$ in \eqref{NeedtoCompare2} also gives 
$$
\lim_{j\rightarrow \infty}\int^{t_1}_{t_0}\int_{\Omega}|v^{k_j}_t(x,s)|^pdxds=\int^{t_1}_{t_0}\int_{\Omega}|v_t(x,s)|^pdxds.
$$
Thus, we are also able to conclude \eqref{SecondConv}. 
\end{proof}

Let us briefly discuss how compactness pertains to the existence of weak solutions.  To this end, assume $\{v^k\}_{k\in \N}$ is the solution sequence
of \eqref{ViscScheme} for a given $\tau>0$. Upon multiplying the PDE in \eqref{ViscScheme} by $v^k - v^{k-1}$ and integrating by parts, we obtain 
$$
\int_{\Omega}\left(\frac{|v^k-v^{k-1}|^p}{\tau^{p-1}} +\frac{1}{p}|Dv^k|^p\right)dx \le \int_{\Omega}\frac{1}{p}|Dv^{k-1}|^pdx, \quad k\in \N.
$$
Moreover, summing over $k=1,\dots, j\in \N$ gives 
\begin{equation}\label{DiscIden}
\sum^{j}_{k=1}\int_{\Omega}\frac{|v^k-v^{k-1}|^p}{\tau^{p-1}}dx + \int_{\Omega}\frac{1}{p}|Dv^j|^pdx\le  \int_{\Omega}\frac{1}{p}|Dg|^pdx,
\end{equation}
which is a discrete analog of the energy identity \eqref{EnergyIdentity}.  

\par Let us further assume $\tau=T/N$ and set $\tau_k=k\tau$ for $k=0,1,\dots, N\in \N$. It will be useful for us to 
define the ``linear interpolating" approximation  as
$$
u_N(\cdot,t):=
v^{k-1}+\left(\frac{t-\tau_{k-1}}{\tau}\right)(v^k-v^{k-1}), 
\quad \tau_{k-1}\le t\le \tau_k, \quad k=1,\dots, N
$$
for $t\in [0,T]$ and $N\in \N$.
It follows from \eqref{DiscIden} that 
\begin{equation}
\int^T_0\int_\Omega |\partial_t u_N(x,t)|^pdxdt + \sup_{0\le t\le T}\int_\Omega |Du_N(x,t)|^pdx \le 2 \int_\Omega |Dg(x)|^pdx\nonumber 
\end{equation}
for all $N\in \N$.
\par Using the ideas given in the proof of Theorem \ref{CompactnessLem}, we obtain a subsequence $(u_{N_j})_{j\in \N}$ and weak solution $u$ of 
\begin{equation}\label{pParabolicT}
\begin{cases}
|u_t|^{p-2}u_t=\Delta_p u, \quad \Omega\times (0,T) \\
\hspace{.5in} u=0, \hspace{.31in} \partial \Omega\times [0,T)\\
\hspace{.5in} u=g, \hspace{.32in} \Omega\times\{0\}
\end{cases}
\end{equation}
for which 
$$
u_{N_j}\rightarrow u
\quad \text{in}\quad \begin{cases}
C([0,T]; L^p(\Omega))\\
L^p([0,T]; W^{1,p}_0(\Omega))
\end{cases}
$$
and
$$
\partial_tu_{N_j}\rightarrow u_t\quad \text{in}\quad L^p(\Omega\times [0,T]).
$$ 
\par For $k\in \N$, let $u^k$ be the weak solution of \eqref{pParabolicT} just described for $T=k$. Moreover, set 
$v^k(\cdot, t)=u^k(\cdot, t)$ for $t\in [0,k]$ and $v^k(\cdot, t):=u^k(\cdot, k)$ for $t\in [k,\infty)$. It is immediate that 
$v^k$ satisfies \eqref{NaturalSpace}.  The proof of Theorem \ref{CompactnessLem} is also readily adapted to give that $(v^k)_{k\in \N}$ has a subsequence converging as in 
\eqref{FirstConv} and \eqref{SecondConv} to a global weak solution $v$ of \eqref{pParabolic}. We omit the details.

\begin{rem} We also remark that the subsequence $(v_{N_j})_{j\in \N}$ of the ``step function" approximation sequence $(v_N)_{N\in \N}$ defined in \eqref{CLaan} converges in $C([0,T]; L^p(\Omega))$ to the same weak solution $u$ as the  sequence $(u_{N_j})_{j\in \N}$. Indeed, by \eqref{DiscIden}
\begin{align*}
\int_{\Omega}|u_{N_j}(x,t)-v_{N_j}(x,t)|^pdx &\le \max_{1\le k\le N_j}\int_{\Omega}|v^k(x)-v^{k-1}(x)|^pdx \\
&\le \left(\frac{T}{N_j}\right)^{p-1}\int_{\Omega}|Dg(x)|^pdx.
\end{align*}
\end{rem}

\section{Large time limit}\label{LargeTlim}
This section is dedicated to the proof of Theorem \ref{LargeTthm}, which details the large time behavior of solutions of the initial value problem \eqref{pParabolic}.  Our main tools are the compactness of weak solutions
of \eqref{pParabolic}  established in Theorem \ref{CompactnessLem} and the following lemma, which involves the sign of weak solutions that are close to $p$-ground states. 

\begin{lem}\label{CrucialLem} For each positive p-ground state $\psi$, $C>0$ and sequence $(s_k)_{k\in \N}$ of positive numbers with $s_k\uparrow\infty$, there is a $\delta=\delta(\psi,C,(s_k)_{k\in \N})>0$ with the following property. If $v$ is a weak solution of \eqref{pParabolic} that satisfies 
\begin{enumerate}[(i)]

\item $\lim_{k\to \infty }e^{\mu_p s_k}v(x,s_k)=\psi$ in $W^{1,p}_0(\Omega)$

\item $\int_\Omega |v(x,0)|^pdx\leq C$

\item $\frac{\int_\Omega |Dv(x,0)|^pdx}{\int_\Omega |v(x,0)|^pdx}\leq \lambda_p+\delta$

\item $\int_\Omega |v^+(x,0)|^pdx\geq\frac{1}{2}\int_\Omega |\psi|^pdx$,
\end{enumerate}
then
\begin{equation}\label{LemmaConclusion}
\int_\Omega |e^{\mu_p t} v^+(x,t)|^pdx\geq \frac{1}{2}\int_\Omega |\psi|^pdx .
\end{equation}
for $t\in [0,1]$.
\end{lem}
\begin{proof} We argue towards a contradiction. If the result fails, then there exists a triplet $(\psi,C,(s_k)_{k\in \N})$ such that for every $\delta>0$, there is a weak solution $v$ that satisfies $(i)-(iv)$ while 
\eqref{LemmaConclusion} fails. Therefore, associated to $\delta_j:=1/j$ $(j\in \N)$, there is a weak solution $v_j$ that satisfies $(i)$,
$$
\int_\Omega |v_j(x,0)|^pdx\leq C, \quad \int_\Omega |v_j^+(x,0)|^pdx\geq\frac{1}{2}\int_\Omega |\psi|^pdx, \quad \frac{\int_\Omega |Dv_j(x,0)|^pdx}{\int_\Omega |v_j(x,0)|^pdx}\leq \lambda_p+\frac{1}{j} 
$$
while 
\begin{equation}\label{JustGonnaSendJ}
\int_\Omega |e^{\mu_p t_j} v_j^+(x,t_j)|^pdx< \frac{1}{2}\int_\Omega |\psi|^pdx 
\end{equation}
for some $t_j\in [0,1]$.  

\par Consequently, the sequence of initial conditions $(v_j(\cdot,0))_{j\in \N}$ is bounded in $W^{1,p}_0(\Omega)$ and has a subsequence (not relabeled) that converges to a positive $p$-ground state $\varphi$ in $W^{1,p}_0(\Omega)$. By Theorem \ref{CompactnessLem}, it also follows that (a subsequence of) the sequence of weak solutions $(v_j)_{j\in \N}$ converges to a weak solution $w$ in $C([0,2], L^p(\Omega))\cap L^p([0,2]; W^{1,p}_0(\Omega))$ with $w(\cdot, 0)=\varphi$. By Corollary \ref{UniqueGround}, $w(\cdot,t)=e^{-\mu_p t}\varphi$. 

\par In addition, we have by $(i)$ and the inequality $\|e^{\mu_p s_k}v_j(\cdot, s_k)\|_{W^{1,p}_0(\Omega)}\le \|v_j(\cdot, 0)\|_{W^{1,p}_0(\Omega)}$ that 
\begin{equation}\label{IneqForContra}
\int_\Omega |Dv_j(x,0)|^pdx\ge \lim_{k\rightarrow\infty}\int_\Omega |D\left(e^{\mu_p s_k}v_j(\cdot, s_k)\right)|^pdx=\int_\Omega |D\psi|^pdx=\lambda_p\int_\Omega |\psi|^pdx
\end{equation}
for all $j\in \N$.  Dividing \eqref{IneqForContra} by $\lambda_p$ and letting $j\rightarrow\infty$ gives 
$$
\int_\Omega|\varphi|^pdx = \frac{1}{\lambda_p}\int_\Omega|D\varphi|^pdx=\frac{1}{\lambda_p}\lim_{j\rightarrow\infty}\int_\Omega |Dv_j(x,0)|^pdx\ge \int_\Omega |\psi|^pdx.\nonumber
$$
However, letting $j\rightarrow\infty$ in \eqref{JustGonnaSendJ} gives 
$$
\int_\Omega |\varphi|^pdx\le  \frac{1}{2}\int_\Omega |\psi|^pdx .
$$
This is a contradiction as $\varphi,\psi\not\equiv 0$. 
\end{proof}

\begin{rem}
A similar conclusion holds for $v^-$ provided $\psi$ is a negative ground state and $(iv)$ is replaced with $\int_\Omega |v^-(x,0)|^pdx\geq\frac{1}{2}\int_\Omega |\psi|^pdx$.
\end{rem}

\begin{proof}[Proof of Theorem \ref{LargeTthm}] We argue in several steps. We first show that for each sequence $(s_k)_{k\in \N}$ of positive numbers with $s_k\uparrow\infty$, a
subsequence of $(e^{\mu_p s_k}v(\cdot, s_k))_{k\in \N}$ has to converge to some $p$-ground state. This in turn will allow us to prove the convergence of the $p$-Rayleigh quotient of $v(\cdot,t)$ to the optimal value $\lambda_p$. Then we will use the convergence of the $p$-Rayleigh quotient of $v(\cdot,t)$ and the sign of this $p$-ground state to derive a crucial lower bound on $L^p(\Omega)$ norm of
the same sign of $e^{\mu_p s_k}v(\cdot, s_k)$. Finally, we use this estimate to show that in fact the full sequence converges to this $p$-ground state.

1. The following limit
\begin{equation}\label{importantLimit}
S:=\lim_{\tau\rightarrow \infty}\int_\Omega |D\left(e^{\mu_p \tau}v(x, \tau)\right)|^pdx
\end{equation}
exists by the monotonicity formula \eqref{ExpDecay}. If $S=0$, we conclude. So let us assume $S>0$, and suppose $(s_k)_{k\in \N}$ is a sequence of positive numbers increasing to $+\infty$. For each $k\in \N$, define
$$
v^k(x,t):=e^{\mu_p s_k}v(x,t+s_k)
$$
for $x\in \Omega$ and $t\ge 0$.

\par Observe, that $v^k$ is a weak solution with $v^k(\cdot,0)=e^{\mu_p s_k}v(\cdot, s_k)$. By \eqref{importantLimit},  $(v^k(\cdot,0))_{k\in \N}\subset W^{1,p}_0(\Omega)$ is a bounded sequence.  By Theorem 2.8,  there is a subsequence $(v^{k_j})_{j\in \N}$ and weak solution $w$ for which $v^{k_j}$ converges to $w$ in $C([0,T]; L^p(\Omega))\cap L^p([0,T], W^{1,p}_0(\Omega))$ for all $T>0$; moreover, $v^{k_j}(\cdot,t)$ converges to $w(\cdot,t)$ 
weakly in $W^{1,p}_0(\Omega)$ for all $t\ge 0$ and strongly for Lebesgue almost every $t\ge 0$. 
\par By \eqref{importantLimit}, we have 
\begin{align*}
S&=\lim_{k\rightarrow\infty}\int_\Omega |D\left(e^{\mu_p (t+s_{k_j})}v(x, t+s_{k_j})\right)|^pdx \\
&=e^{p\mu_p t}\lim_{j\rightarrow\infty}\int_\Omega |Dv^{k_j}(x, t)|^pdx \\
&=e^{p\mu_p t}\int_\Omega |Dw(x, t)|^pdx
\end{align*}
for almost every $t\ge 0$. However, as $t\mapsto \int_\Omega |Dw(x, t)|^pdx$ is absolutely continuous (by Lemma \ref{AbsPhi}), this equality holds for every $t\ge 0$.   Moreover, it must be that $\lim_{j\rightarrow\infty}\int_\Omega |Dv^{k_j}(x, t)|^pdx = \int_\Omega |Dw(x, t)|^pdx$ also holds for each $t\ge 0$.

\par 2. In addition, we have 
\begin{align}\label{yesGroundState}
0&=\frac{d}{dt}e^{p\mu_p t}\int_\Omega |Dw(x, t)|^pdx \nonumber \\
&=pe^{p\mu_p t}\left\{\mu_p\int_\Omega |Dw(x, t)|^pdx- \int_\Omega |w_t(x, t)|^pdx\right\}
\end{align}
for almost every $t\ge 0$. This computation follows from Lemma \ref{AbsPhi}. By the proof of Lemma \ref{LemmaINeq}, $\mu_p\int_\Omega |Dw(x, t)|^pdx\le \int_\Omega |w_t(x, t)|^pdx$  for almost every $t\ge 0$ and equality holds only if $w(\cdot,t)$ is a $p$-ground state for almost every $t\ge 0$.  Since $t\mapsto w(\cdot,t)\in L^p(\Omega)$ is continuous and $S=\int_\Omega |D\left(e^{\mu_p t}w(x, t)\right)|^pdx$, there is a single $p$-ground state $\psi$ for which 
$$
w(\cdot,t)=e^{-\mu_p t}\psi.
$$
\par In summary,
\begin{equation}\label{SScaledLimit}
\lim_{j\rightarrow \infty}e^{\mu_p (t+s_{k_j})}v(\cdot, t+s_{k_j})=\psi
\end{equation}
in $W^{1,p}_0(\Omega)$ for each $t\ge 0$ and in $L^p(\Omega)$ uniformly for each interval $0\le t\le T$.  Moreover, 
$$
\lim_{t\rightarrow\infty}\frac{\int_\Omega|Dv(x,t)|^pdx}{\int_\Omega|v(x,t)|^pdx}=\lim_{j\rightarrow\infty}\frac{\int_\Omega|D(e^{\mu_ps_{k_j}}v(x,s_{k_j}))|^pdx}{\int_\Omega|e^{\mu_ps_{k_j}}v(x,s_{k_j})|^pdx}=
\frac{\int_\Omega|D\psi|^pdx}{\int_\Omega|\psi|^pdx}=\lambda_p.
$$

\par 3. As $S=\int_\Omega|D\psi|^pdx>0$, the $p$-ground state $\psi$ is determined by its sign. Let us first assume $\psi$ is positive and choose $\delta=\delta(\psi, C, (s_{k_j})_{j\in \N})$ as in Lemma 
\ref{CrucialLem} where 
$$
C:=\frac{1}{\lambda_p}\int_{\Omega}|Dv(x,0)|^pdx. 
$$
Note by Poincar\'{e}'s inequality \eqref{PoincareIneq} and Lemma \ref{LemmaINeq}
\begin{equation}\label{AlwaysLessThanC0}
\int_\Omega |e^{\mu_p t}v(x, t)|^pdx\le \frac{1}{\lambda_p}\int_\Omega |D\left(e^{\mu_p t}v(x, t)\right)|^pdx\le C
\end{equation}
for all $t\ge 0$. 

\par Now fix $j_0\in \N$ so large that 
$$
 \int_\Omega |(v^{k_{j}})^+(x,0)|^pdx\geq\frac{1}{2}\int_\Omega |\psi|^pdx \quad \text{and}\quad \frac{\int_\Omega |Dv^{k_{j}}(x,0)|^pdx}{\int_\Omega |v^{k_{j}}(x,0)|^pdx}\leq \lambda_p+\delta
$$
for $j\ge j_0$.  Let us additionally fix an $\ell\ge j_0$.  By \eqref{AlwaysLessThanC0}
$$
\int_\Omega |v^{k_{\ell }}(x, 0)|^pdx=\int_\Omega |e^{\mu_p s_{k_{\ell }}}v(x, s_{k_{\ell }})|^pdx\le C,
$$
and by \eqref{SScaledLimit}, 
$$
\lim_{j\rightarrow\infty}e^{\mu_p s_{k_j}}v^{k_{\ell }}(\cdot,s_{k_j})=\lim_{j\rightarrow\infty}e^{\mu_p(s_{k_j}+s_{k_{\ell }})}v(\cdot,s_{k_j}+s_{k_{\ell }})=\psi
$$
in $W^{1,p}_0(\Omega)$.   Lemma \ref{CrucialLem} then implies 
\begin{equation}\label{NeedthisatT1}
\int_\Omega |e^{\mu_p t} (v^{k_{\ell }})^+(x,t)|^pdx\geq \frac{1}{2}\int_\Omega |\psi|^pdx 
\end{equation}
or equivalently, 
\begin{equation}\label{GotAtT1}
\int_\Omega \left|e^{\mu_p \left(t + s_{k_{\ell }}\right)}v^+(x,t+s_{k_{\ell }})\right|^pdx\geq \frac{1}{2}\int_\Omega |\psi|^pdx .
\end{equation}
for $0\le t\le 1$. 

\par Now set 
$$
u(x,t):=e^{\mu_p}v^{k_{\ell }}(x,t+1)=e^{\mu_p\left(1+s_{k_{\ell }}\right)}v(x,t+1+s_{k_{\ell }})
$$
for $x\in \Omega$ and $t\ge 0.$ Let us verify the hypotheses of Lemma \ref{CrucialLem}. 

\begin{enumerate}[$(i)$]

\item In view of \eqref{SScaledLimit}, 
$$
\lim_{j\rightarrow\infty}e^{\mu_p s_{k_j}}u(\cdot,s_{k_j})=\lim_{j\rightarrow\infty}e^{\mu_p\left(1+s_{k_{\ell }}+s_{k_j}\right)}v(\cdot,1+s_{k_{\ell }}+s_{k_j} )=\psi
$$
in $W^{1,p}_0(\Omega)$. 

\item By \eqref{AlwaysLessThanC0},
$$
\int_\Omega |u(x, 0)|^pdx=\int_\Omega \left|e^{\mu_p\left(1+s_{k_{\ell }}\right)}v(x,1+s_{k_{\ell }})\right|^pdx\le C.
$$

\item By Proposition \ref{RayleighGoDown},

\begin{align*}
\frac{\int_\Omega|Du(\cdot,0)|^pdx}{\int_\Omega|u(\cdot,0)|^pdx} &= \frac{\int_\Omega\left|Dv(x,1+s_{k_{\ell }})\right|^pdx}{\int_\Omega\left|v(x,1+s_{k_{\ell }})\right|^pdx} \\
&\le  \frac{\int_\Omega\left|Dv(x,s_{k_{\ell }})\right|^pdx}{\int_\Omega\left|v(x, s_{k_{\ell }})\right|^pdx}\\
&=\frac{\int_\Omega\left|Dv^{k_{\ell }}(x,0)\right|^pdx}{\int_\Omega\left|v^{k_{\ell }}(x,0)\right|^pdx}\\
&\le \lambda_p +\delta.
\end{align*}

\item Evaluating \eqref{NeedthisatT1} at $t=1$ gives
$$
\int_\Omega |u^+(x,0)|^pdx=\int_\Omega |e^{\mu_p } (v^{k_{\ell }})^+(x,1)|^pdx\geq \frac{1}{2}\int_\Omega |\psi|^pdx.
$$
\end{enumerate}
Then Lemma \ref{CrucialLem} implies 
$$
\int_\Omega |e^{\mu_p t} u^+(x,t)|^pdx= 
\int_\Omega |e^{\mu_p(t+1)}\left(v^{k_{\ell }}\right)^+(x,t+1)|^pdx\geq \frac{1}{2}\int_\Omega |\psi|^pdx .
$$
for $0\le t\le 1$. Combining with \eqref{GotAtT1} we have that in fact \eqref{GotAtT1} holds for $0\le t\le 2$, and by induction, it holds for all $t\ge 0.$   Therefore, 
\begin{equation}\label{GeneralvPlusBound}
\int_\Omega \left|e^{\mu_p \left(t + s_{k_{j}}\right)}v^+(x,t+s_{k_{j}})\right|^pdx\geq \frac{1}{2}\int_\Omega |\psi|^pdx, \quad t\in[0,\infty)
\end{equation}
for $j\ge j_0$. Finally, if $\psi<0$, inequality \eqref{GeneralvPlusBound} holds with $v^-$ replacing $v^+$. 

\par 4. Now let $(t_\ell)_{\ell\in \N}$ be another sequence of positive numbers increasing to infinity. From our arguments above, $t_\ell$ has a subequence 
(that we won't relabel) such that $e^{\mu_p t_\ell}v(\cdot,t_\ell)$ converges to a $p$-ground state $\varphi$ in $W^{1,p}_0(\Omega)$ as $\ell\rightarrow \infty$. Moreover, 
$\varphi$ also satisfies $S=\int_\Omega|D\varphi|^pdx$.  By the simplicity of $\lambda_p$, $\varphi=\psi$ or $\varphi=-\psi$. Let us assume $\varphi=-\psi$ and without any loss of generality, $\varphi<0$. 
As $t_\ell$ is increasing, we can choose a subsequence $(t_{\ell_j})_{j\in \N}$ for which 
$$
t_{\ell_j}>s_{k_j},\quad \quad j\in \N. 
$$
Substituting $t=t_{\ell_j}-s_{k_j}>0$ in \eqref{GeneralvPlusBound} gives, 
$$
\int_\Omega \left|e^{\mu_p t_{\ell_j}}v^+(x,t_{\ell_j})\right|^pdx\geq \frac{1}{2}\int_\Omega |\psi^+|^pdx, \quad j\in\N.
$$
However, after letting $j\rightarrow\infty$ we find 
$$
\int_\Omega \left|\varphi^+\right|^pdx\geq \frac{1}{2}\int_\Omega |\psi^+|^pdx
$$
which cannot occur since $\varphi<0$ and $\psi>0$ in $\Omega$. 

\par Consequently, for every sequence $(s_k)_{k\in \N}$ of positive numbers increasing to $\infty$, there is a subsequence of $(e^{\mu_p s_k}v(\cdot, s_k))_{k\in \N}$ converging in $W^{1,p}_0(\Omega)$
to a $p$-ground state $\psi$ with the same sign that satisfies $S=\int_\Omega|Dw|^pdx$. We appeal to the simplicity of $\lambda_p$ once again to conclude there is only one such ground state $\psi$. Therefore, $\lim_{t\rightarrow\infty}e^{\mu_pt}v(\cdot, t)=\psi$ 
in $W^{1,p}_0(\Omega)$, as asserted. 
\end{proof}
\begin{rem}
By Morrey's inequality, the family $\{e^{\mu_p t}v(\cdot, t)\}_{t\ge 0}$ is precompact in $C^{0,1-n/p}(\Omega)$ for $p>n$. In this case, $\lim_{t\rightarrow\infty}e^{\mu_pt}v(x, t)=\psi(x)$ uniformly in $x\in \Omega$.  It would be of great interest to establish uniform convergence for all $p>1$.  It seems to us that the lacking 
piece of information is a modulus of continuity estimate on solutions of \eqref{MainPDE}. Indeed, we have not succeeded in deriving any useful a priori estimates on solutions of 
\eqref{MainPDE}.  We hope to do so in forthcoming work. 
\end{rem}

\section{Viscosity solutions}\label{ViscSoln}
We now turn our attention to proving Theorem \ref{ImplicitSchemeThm}. Therefore, we assume throughout this section that $p\ge 2$, $g\in W^{1,p}_0(\Omega)\cap C(\overline{\Omega})$, and that there is a $p$-ground state $\varphi$ for which 
\begin{equation}\label{gPsiIneq}
-\varphi(x)\le g(x)\le \varphi(x),  \quad x\in \overline{\Omega}.
\end{equation}
These assumptions will help us verify that \eqref{pParabolic} has a unique {\it viscosity solution} that is also a weak solution; the reader can find important background material on the theory 
of viscosity solutions from sources such as \cite{BC, CIL, FS}. We remark that we do not consider the ``singular'' case $p\in (1,2)$ in order to avoid technicalities 
and to focus on the new ideas needed to build viscosity solutions of \eqref{pParabolic}.  

\par While establishing the uniqueness of viscosity solutions of the initial value problem \eqref{pParabolic} is far from trivial, a standard proof for the comparison of viscosity solutions $(p=2)$ of the heat equation is 
readily adapted to \eqref{pParabolic}. For instance, it is possible to modify the proofs of Theorem 8.2 of \cite{CIL}, Theorem 8.1 of section V.8 in \cite{FS}, or Theorem 4.7 of \cite{JLM2}  to prove the following proposition.
The main feature to be exploited is that the term $|v_t|^{p-2}v_t$ is strictly increasing in the time derivative $v_t$. 
\begin{prop}\label{UsualComparison}
Assume $v\in USC(\overline{\Omega}\times[0,T))$ and $w\in LSC(\overline{\Omega}\times[0,T))$. Suppose the inequality
$$
|v_t|^{p-2}v_t-\Delta_pv\le 0\le |w_t|^{p-2}w_t-\Delta_pw, \quad \Omega\times(0,T)
$$
holds in the sense of viscosity solutions and $v(x,t)\le w(x,t)$ for $(x,t)\in \partial\Omega\times[0,T)$ and for $(x,t)\in \Omega\times\{0\}$. Then 
$$
v\le w
$$
in $\Omega\times(0,T).$
\end{prop}
 Consequently, we will 
concentrate on confirming the existence of a viscosity solution and showing that this solution is indeed a weak solution.  Fortunately, we propose a method that resolves both issues simultaneously.  Let us first begin by observing that solutions of the implicit time scheme \eqref{ViscScheme} generate viscosity solutions. 

\begin{lem}\label{SchemeVISC} For each $\tau>0$,  the implicit scheme \eqref{ViscScheme} generates a solution sequence $\{v^k\}$
of viscosity solutions. Moreover,
$$
\sup_{\Omega}|v^k|\le\sup_{\Omega}|g|
$$
and $v^k\in C^{1,\alpha}_{\text{loc}}(\Omega)$ for some $\alpha\in (0,1]$ and each $k\in \N$.
\end{lem}

\begin{proof}
Consider the implicit scheme \eqref{ViscScheme} for $k=1$
\begin{equation}\label{ViscSchemeOne} 
{\cal J}_p\left(\frac{v^1 - g}{\tau}\right)=\Delta_pv^1,\quad x\in \Omega.
\end{equation}
As ${\cal J}_p$ is increasing, this PDE admits a comparison principle for weak sub- and supersolutions.  Since the constant 
function $\sup_{\Omega}|g|$ is a supersolution, that is nonnegative on $\partial \Omega$, $v^1\le \sup_{\Omega}|g|$. Likewise, 
$v^1\ge -\sup_{\Omega}|g|$, and thus $|v^1|\le \sup_{\Omega}|g|$. As the left hand side of the PDE \eqref{ViscSchemeOne} is now identified as an $L^\infty(\Omega)$ function, Theorem 2 in \cite{DB} implies there is some $\alpha\in(0,1]$ such that $v^1\in C^{1,\alpha}_{\text{loc}}(\Omega)$.  The assertion for each $v^k$ follows routinely by induction.

\par Let us now verify that $v^1$, and similarly each $v^k$, is a viscosity solution. We will closely follow the argument used to prove Theorem 2.5 in \cite{JLM2}.  Assume that 
$v^1-\phi$ has a strict local minimum at $x_0$ where $\phi\in C^\infty(\Omega)$. We are to show 
\begin{equation}\label{visc4v1}
{\cal J}_p\left(\frac{v^1(x_0) - g(x_0)}{\tau}\right)\ge \Delta_p\phi(x_0).
\end{equation}
If \eqref{visc4v1} doesn't hold, there is a $\delta>0$ where
$$
\begin{cases}
{\cal J}_p\left(\frac{v^1(x) - g(x)}{\tau}\right)< \Delta_p\phi(x),\\
(v^1-\phi)(x)>(v^1-\phi)(x_0)
\end{cases}
$$
for $x\in B_\delta(x_0).$ Set 
$$
c:=\inf_{\partial B_\delta(x_0)}(v^1-\phi)>(v^1-\phi)(x_0),
$$
and observe
$$
-\Delta_p(\phi+c)<-{\cal J}_p\left(\frac{v^1 - g}{\tau}\right)=-\Delta_pv^1, \quad  x\in B_\delta(x_0)\\
$$
while $\phi+c\le v^1$ for $x\in \partial B_\delta(x_0)$. By comparison, $\phi+c\le v^1$ in $\overline{B}_\delta(x_0)$. In particular, 
$$
c\le (v^1-\phi)(x_0)
$$
which is a contradiction. Hence, \eqref{visc4v1} holds and the argument for the subsolution property of $v^1$ can be made similarly. 
\end{proof}

\begin{cor}\label{LemdiscreteVisc}
Let $\tau>0$.  Assume $\{\psi^k\}^\infty_{k=0}\subset C^\infty(\Omega)$ and $(x_0,k_0)\in\Omega\times\N$ is such that
\begin{equation}\label{discreteVisc}
v^k(x)-\psi^k(x)\le v^{k_0}(x_0)-\psi^{k_0}(x_0) 
\end{equation}
for $x$ in a neighborhood of $x_0$ and $k\in \{k_0-1,k_0\}$. Then 
$$
{\cal J}_p\left(\frac{\psi^{k_0}(x_0) - \psi^{k_0-1}(x_0)}{\tau}\right)\le \Delta_p\psi^{k_0}(x_0).
$$
\end{cor}

\begin{proof}
Evaluating the left hand side \eqref{discreteVisc} at $k=k_0$ gives 
$$
{\cal J}_p\left(\frac{v^{k_0}(x_0) - v^{k_0-1}(x_0)}{\tau}\right)\le \Delta_p\psi^{k_0}(x_0)
$$
as $v^{k_0}$ is a viscosity solution of \eqref{ViscScheme}.  Evaluating the left hand side of \eqref{discreteVisc} at $x=x_0$ and $k=k_0-1$ gives
$\psi^{k_0}(x_0) - \psi^{k_0-1}(x_0)\le v^{k_0}(x_0) - v^{k_0-1}(x_0)$. The claim follows from the above inequality and the monotonicity of 
${\cal J}_p$. 
\end{proof} 
Our candidate for a viscosity solution of \eqref{pParabolic} is $\lim_{N\rightarrow\infty} v_N$ where $v_N$ is defined in \eqref{CLaan}.  We have 
already established that a subsequence of $(v_N)_{N\in \N}$ converges to a weak solution in $C([0,T]; L^p(\Omega))$. Therefore, we are left to verify that this 
sequence converges uniformly to a viscosity solution. Towards this goal, we will employ the half-relaxed limits of $v_N$ 
$$
\overline{v}(x,t):=\limsup_{\substack{N\rightarrow\infty\\ (y,s)\rightarrow(x,t)}}v_N(y,s)
$$
$$
\underline{v}(x,t):=\liminf_{\substack{N\rightarrow\infty\\ (v,s)\rightarrow(x,t)}}v_N(y,s)
$$
for $x\in \overline\Omega$ and $t\in [0,T]$.
\par By Lemma \ref{SchemeVISC}, the sequence $\{v_N\}_{N\in \N}$ is bounded, independently of $N\in \N$.  As a result, the above functions are well defined and finite at each 
$(x,t)\in \overline\Omega\times [0,T]$. Moreover, $\overline{v}, -\underline{v}$ are upper semicontinuous and $\overline{v}=\underline{v}$ 
if and only if $v_N$ converges locally uniformly (see Remark 6.4 of \cite{CIL}).  It is immediate that $\underline v\le \overline v$. In order to conclude $\overline v\le \underline v$, we will show that
$\underline{v}(x,t)=\overline{v}(x,t)$ when $t=0$ and when $x\in \partial \Omega$ and that $\overline v$ and $\underline v$ are respective viscosity sub- and supersolutions of the PDE \eqref{MainPDE}. We would then be in a position to apply Proposition \ref{UsualComparison}.

\begin{lem}\label{BoundaryBarrier}
Let $\varphi$ be the $p$-ground state in \eqref{gPsiIneq}. Then for $N\in\N$
\begin{equation}\label{BoundaryEstimate}
-\varphi(x)\le v_N(x,t)\leq \varphi(x),\quad (x,t)\in \overline\Omega\times[0,T].
\end{equation}
In particular, for $x_0\in\partial \Omega$, $\overline v(x_0,t)=\underline v(x_0,t)=0$.
\end{lem}
\begin{proof} Observe that
$$
-\lap_p \varphi+{\cal J}_p\left(\frac{\varphi - g}{\tau}\right)=\lambda_p|\varphi|^{p-2}\varphi+{\cal J}_p\left(\frac{\varphi - g}{\tau}\right)\geq 0
$$
in $\Omega$ as $\varphi\ge 0$ and $\varphi\ge g$.
Therefore, $\varphi$ is a supersolution of \eqref{ViscSchemeOne}. Since $\varphi=v^1=0$ on $\partial \Omega$, weak comparison implies $v^1\leq \varphi.$
Likewise, $v^1\geq - \varphi.$  Iterating these bounds for each $k$ yields $-\varphi\le v^k\leq \varphi$. Consequently, \eqref{BoundaryEstimate} holds. Since 
$\partial\Omega$ is smooth, we have that $\varphi\in C(\overline\Omega)$ \cite{Saka}. Thus, for $x_0\in\partial\Omega$ and $t\in [0,T]$, we can pass to the limit 
in \eqref{BoundaryEstimate} to conclude $\overline v(x_0,t)\le 0\le\underline v(x_0,t)$.
\end{proof}

\begin{lem}\label{zeroBarrier}
For each $x_0\in \Omega$ and $\e>0$, there is a constant $C=C(x_0,\e)$ such that
\begin{equation}\label{BarrierTzero}
|v_N(x,t)-g(x_0)|\le \e+C\left(t+\frac{T}{N}+|x-x_0|^\frac{p}{p-1}\right)
\end{equation}
for $(x,t)\in \overline\Omega\times[0,T]$ and $N\in \N$. In particular, $\overline v(x_0,0)=\underline v(x_0,0)=g(x_0).$
\end{lem}
\begin{proof}
We first prove there is a constant $C=C(x_0,\e)$ for which   
$$
u(x):=g(x_0)+\e+C\left(\tau+c_p|x-x_0|^\frac{p}{p-1}\right), \quad x\in \overline\Omega
$$
lies above $v^1$. Here $c_p$ is selected so that $\lap_p\left( c_p|x-x_0|^\frac{p}{p-1}\right)=1$. Note that since $g$ is continuous on $\overline\Omega$, we can find a $\delta>0$ and $C>0$ so that
$$
|g(x)-g(x_0)|<\e
$$
when $|x-x_0|<\delta$
and
$$
\sup_\Omega |g|\leq C c_p|x-x_0|^\frac{p}{p-1}
$$
when $|x-x_0|\geq \delta$.   Indeed, we may choose 
$$
C=\frac{2\sup_\Omega|g|}{c_p\delta^\frac{p}{p-1}}.
$$

\par By design, $g(x_0)+\e+Cc_p|x-x_0|^\frac{p}{p-1}-g(x)\geq 0$ for all $x\in\Omega$. Therefore,
\begin{align*}
-\lap_p u+{\cal J}_p\left(\frac{u - g}{\tau}\right)&=-C^{p-1}+{\cal J}_p\left(\frac{g(x_0)+\e+Cc_p|x-x_0|^\frac{p}{p-1}-g+C\tau}{\tau}\right)\\
&\geq -C^{p-1}+C^{p-1}\\
&=0. 
\end{align*}
Choosing $C$ even larger if necessary, we may also assume that $u\geq 0$ on $\partial \Omega$.  In this case, weak comparison gives
$$
v^1(x)\leq u(x)=g(x_0)+\e+C\left(\tau+c_p|x-x_0|^\frac{p}{p-1}\right)
$$
$x\in\Omega$. Similarly, we have
$$
v^1(x)\geq g(x_0)-\e-C\left(\tau+c_p|x-x_0|^\frac{p}{p-1}\right).
$$
\par After iterating this procedure $k$ times, we find
$$
g(x_0)-\e-C\left(k\tau+c_p|x-x_0|^\frac{p}{p-1}\right)\leq v^k(x)\leq g(x_0)+\e+C\left(k\tau+c_p|x-x_0|^\frac{p}{p-1}\right).
$$
By the definition of $v_N$ (in which $\tau=T/N$), we obtain for $t\in ((k-1)T/N, kT/N]$
$$
v_N(x,t)=v^k(x)\leq g(x_0)+\e+C\left(t+\frac{T}{N}+c_p|x-x_0|^\frac{p}{p-1}\right).
$$
The analogous lower bound holds as well, which implies \eqref{BarrierTzero}.  As a result
$$
g(x_0)-\e-C\left(t+c_p|x-x_0|^\frac{p}{p-1}\right)\leq \underline v(x,t)\leq \overline v(x,t)\leq g(x_0)+\e+C\left(t+c_p|x-x_0|^\frac{p}{p-1}\right),
$$
and therefore
$$
g(x_0)-\e\leq \underline v(x_0,0)\leq \overline v(x_0,0)\leq g(x_0)+\e.
$$
These inequalities conclude the proof, as $\e>0$ is arbitrary.
\end{proof}

\par The following lemma will allow us to exploit the 
discrete viscosity solutions property of solutions sequences of \eqref{ViscScheme} as described in Corollary \ref{LemdiscreteVisc}. We note this
statement is an analog of Lemma A.3 in \cite{BarPer} and is inspired by other works of G. Barles and B. Perthame such as \cite{BarPer2}.
\begin{lem}\label{ApproxPoints}
Assume $\phi\in C^\infty(\Omega\times(0,T))\cap C(\overline\Omega\times[0,T]) $.  For $N\in \N$ define 
$$
\phi_N(x,t):=
\begin{cases}
\phi(x,0), \quad & (x,t)\in \Omega\times\{0\},\\
\phi(x,\tau_k), \quad & (x,t)\in\Omega \times (\tau_{k-1}, \tau_k]
\end{cases}, \quad k=1,\dots,N.
$$
Suppose $\overline{v}-\phi$ ($\underline{v}-\phi$) has a strict local maximum (minimum) at $(x_0,t_0)\in \Omega\times(0,T)$. Then there are sequences $(x_j,t_j)\rightarrow  (x_0,t_0)$ and $N_j\rightarrow \infty$, as $j\rightarrow \infty$, such that 
$v_{N_j} -\phi_{N_j}$ has local maximum (minimum) at $(x_j,t_j)$ for each $j\in \N$. 
\end{lem}

\begin{proof}
First note that $\phi_N$ converges to $\phi$ uniformly on $\Omega\times [0,T]$. Thus, 
$$
(\overline{v}-\phi)(x,t):=\limsup_{\substack{N\rightarrow\infty\\ (y,s)\rightarrow(x,t)}}(v_N-\phi_N)(y,s).
$$
Consequently, without of loss of generality, we may prove the claim for $\phi\equiv 0$.  Another important observation for us is that for any nonempty, compact subset $D\subset\Omega$ and 
any nonempty, subinterval $I\subset [0,T]$, $v_N$ will achieve a maximum value on $D\times I$.  This follows from the continuity of $v^k$ as
\begin{equation}\label{MaxFormUN}
\sup_{D\times I}v_N=\max\left\{\max_{D}v^k(x): k=1,\dots, N\; \text{such that}\;I\cap (\tau_{k-1},\tau_k]\neq \emptyset\right\}.
\end{equation}

\par Now assume that there is $r>0$ such that  
\begin{equation}\label{StrictMaxubar}
\overline{v}(x,t)<\overline{v}(x_0,t_0), \quad (x,t)\in Q_r,
\end{equation}
where $Q_r:=B_r(x_0)\times (t_0-r,t_0+r)\subset \Omega\times(0,T)$.  By definition, we may select a maximizing sequence  $\overline{v}(x_0,t_0)=\lim_{j\rightarrow \infty}v^{N_j}(y_j,s_j)$ where $(y_j,s_j)\rightarrow (x_0,t_0)$ and $N_j\rightarrow \infty$. Without loss of generality, 
we may assume $(y_j,s_j)\in Q_r$ for all $j\in \N$.  By the equality \eqref{MaxFormUN}, we may assume there is an $(x_j,t_j)\in \overline{Q_r}$ maximizing $v_{N_j}$ over $Q_r$. By compactness,  we may also assume 
that up to a subsequence $(x_j, t_j)\rightarrow (x_1,t_1)\in \overline{Q_r}$ as $j\rightarrow \infty$. 
\par Hence, 
\begin{align*}
\overline{v}(x_0,t_0) & =\limsup_{j\rightarrow \infty}v^{N_j}(y_j,s_j)\\
& \le \limsup_{j\rightarrow \infty}v^{N_j}(x_j,t_j)\\
& \le \overline{v}(x_1,t_1).
\end{align*}
By \eqref{StrictMaxubar}, $(x_1,t_1)=(x_0,t_0)$ and the claim follows. 
\end{proof}

\begin{proof}[Proof of Theorem \ref{ImplicitSchemeThm}] It suffices to show that $\overline{v}$ is a viscosity subsolution 
and $\underline{v}$ is a supersolution of \eqref{MainPDE}. Recall that Lemmas \ref{BoundaryBarrier} and \ref{zeroBarrier} assert that $\overline{v}$ and $\underline v$ agree on $\partial \Omega$ and at $t=0$, which would allow us to apply Proposition \ref{UsualComparison} to conclude 
$\overline v\le \underline v$.  In this case, $\overline v=\underline v$ and $v_N\rightarrow v$ uniformly in $\overline{\Omega}\times[0,T]$.  
 \par  Assume that $\phi\in C^\infty(\Omega\times(0,T))$ and $\overline{v}-\phi$ has a strict local maximum at $(x_0,t_0)\in \Omega\times(0,T)$. By Lemma \ref{ApproxPoints}, there are points 
$(x_j,t_j)$ converging to $(x_0,t_0)$ and $N_j\in \N$ tending to $+\infty$, as $j\rightarrow \infty$, such that $v_{N_j}-\phi_{N_j}$ has a local maximum at $(x_j,t_j)$.  Observe that for each $j\in \N$, 
$t_j\in (\tau_{k_j-1}, \tau_{k_j}]$ for some $k_j\in\{0,1,\dots, N_j\}$. Hence, by the definition of $v_{N_j}$ and $\phi_{N_j}$, 
$$
\Omega\times \{0,1,\dots,N_j\}\ni (x,k)\mapsto v^k(x)- \phi(x,\tau_{k})
$$
has a local maximum at $(x,k)=(x_j, k_j)$.  By Lemma \ref{LemdiscreteVisc}, 
$$
{\cal J}_p\left(\frac{ \phi(x_j,\tau_{k_j})-  \phi(x_j,\tau_{k_j-1})}{T/N_j}\right)\le \Delta_p \phi(x_j,\tau_{k_j}).
$$
\par As $\tau_{k_j-1}=\tau_{k_j}-T/N_j$ and $|t_j-\tau_{k_j}|\le T/N_j$ for $j\in \N$, we can appeal to the smoothness of $\phi$ and send $j\rightarrow \infty$ to arrive at 
$$
{\cal J}_p(\phi_t(x_0,t_0))\le \Delta_p\phi(x_0,t_0).
$$
Consequently, $\underline v$ is a viscosity subsolution of \eqref{MainPDE}.  By the homogeneity of equation \eqref{MainPDE}, the same argument applied to $-\overline{v}$ yields that $\underline{v}$ is a supersolution. 
\end{proof}
We conclude this section by arguing that when $g\in C^2(\overline{\Omega})$, viscosity solutions of \eqref{pParabolic} satisfy
$x\mapsto v(x,t) \in C^{1,\alpha}_\text{loc}(\Omega)$ for almost every  $t>0$ and $|v_t|\le C$.
\begin{prop}\label{GCprop}
Assume $v$ is a viscosity solution of \eqref{pParabolic} and that there is a constant $C\ge 0$ such that 
\begin{equation}\label{Gcond}
|C|^{p-2}C\ge \Delta_pg(x), \quad x\in \Omega.
\end{equation}
Then for each $t\ge s$ and $x\in \Omega$
$$
v(x,t)\le v(x,s) + C(t-s).
$$
In particular $v_t\le C$. Likewise, if $v$ is a viscosity solution of \eqref{pParabolic} and there is $C\le 0$ such that 
$$
|C|^{p-2}C\le \Delta_pg(x), \quad x\in \Omega.
$$
Then $v_t\ge C$.
\end{prop}
\begin{proof}
By assumption \eqref{Gcond}, $(x,t)\mapsto g(x)+ Ct$ is a supersolution of \eqref{MainPDE} that is at least as large as $v$ on $\partial \Omega$ and when $t=0$. By Proposition \ref{UsualComparison}, 
$v(x,t)\le g(x) +Ct$. Now assume $\tau>0$ is fixed and set $w_1(x,t):=v(x,t+\tau)$ and $w_2(x,t):=v(x,t)+C\tau$. Observe that $w_1$ and $w_2$ are viscosity solutions of \eqref{MainPDE} and $w_1(x,t)\le w_2(x,t)$ when either $(x,t)\in \partial\Omega\times[0,T)$ or when $x\in \Omega$ and $t=0$. By 
Proposition \ref{UsualComparison}, $w_1\le w_2$ and so $v(x,t+\tau)\le v(x,t)+C\tau$. We may argue similarly for the other assertion.
\end{proof}

\begin{cor}
Assume $v$ is a viscosity solution of \eqref{pParabolic} and there is $C\ge 0$ such that 
\begin{equation}\label{Gcond2}
|C|^{p-2}C\ge |\Delta_pg(x)|, \quad x\in \Omega. 
\end{equation}
Then $|v_t|\le C$. 
\end{cor}
\begin{cor}
Assume $v$ is a viscosity solution of \eqref{pParabolic} and $g$ satisfies \eqref{Gcond2} for some $C\ge 0$. Then for almost every $t \ge 0$, 
$x\mapsto v(x,t)\in C^{1,\alpha}_\text{loc}(\Omega)$. 
\end{cor}
\begin{proof}
As $\Delta_p  v=|v_t|^{p-2}v_t \in L^\infty(\Omega)$, for almost every $t>0$, the claim follows from Theorem 2 in \cite{DB}.  
\end{proof}

\section{Large $p$ limit}\label{LargePlim}
We are now prepared to deduce the large $p$ limit of equation \eqref{MainPDE} and prove Theorem \ref{pinfProp}. We interpret this assertion as a parabolic analog of a 
theorem of P. Juutinen, P. Lindqvist and J. Manfredi \cite{JLM}. We also encourage the reader to compare this Theorem \ref{pinfProp} with the results of \cite{JL}.  

\begin{proof}[Proof of Theorem \ref{pinfProp}]
By \eqref{EnergyIdentity} and the assumption that $g\in W^{1,\infty}_0(\Omega)$,  $(v^p_t)_{p>r}$ and $(Dv^p)_{p>r}$ are bounded in $L^{r}_{\text{loc}}(\Omega\times(0,\infty))$ for each $r\ge 1$.  Morrey's inequality then implies $(v^p)_{p>n+1}\subset C^{1-(n+1)/p}_{\text{loc}}(\Omega\times(0,\infty))$ has a subsequence $(v^{p_k})_{k\in \N}$ that converges locally uniformly to a continuous function $v$ on $\Omega\times(0,\infty)$. Now suppose $\phi\in C^\infty(\Omega\times (0,\infty))$ and $v-\phi$ has a strict local maximum at some $(x_0,t_0)\in 
\Omega\times (0,\infty)$. We aim to show \newline $G_\infty(\phi_t(x_0,t_0),D\phi(x_0,t_0), D^2\phi(x_0,t_0))\le 0$; that is,
\begin{equation}\label{WantInfPara}
0\ge 
\begin{cases}
\min\{-\Delta_\infty\phi(x_0,t_0), |D\phi(x_0,t_0)| + \phi_t(x_0,t_0)\}, \quad &\phi_t(x_0,t_0)<0\\
-\Delta_\infty\phi, \quad & \phi_t(x_0,t_0)=0\\
\max\{-\Delta_\infty\phi(x_0,t_0),- |D\phi(x_0,t_0)| + \phi_t(x_0,t_0)\}, \quad& \phi_t(x_0,t_0)>0\\
\end{cases}.
\end{equation}
By the uniform convergence of $v^{p_k}$ to $v$, there is a sequence of points $(x_k,t_k)\rightarrow (x_0,t_0)$ such 
that $v^{p_k}-\phi$ has a local maximum at $(x_k,t_k)$. As $v^{p_k}$ is a viscosity solution of \eqref{MainPDE}, 
\begin{equation}\label{phikCond}
|\phi_t(x_k,t_k)|^{p_k-2}\phi_t(x_k,t_k)\le \Delta_{p_k}\phi(x_k,t_k), \quad k\in \N. 
\end{equation}
\par If $\phi_t(x_0,t_0)>0$, then $\phi_t(x_k,t_k)>0$ for all $k$ large enough. Moreover, \eqref{phikCond} implies $-\Delta_{p_k}\phi(x_k,t_k)<0$ and $|D\phi(x_k, t_k)|\neq 0$ for all $k$ large. Rearranging \eqref{phikCond} gives 
\begin{equation}\label{PtendInf}
\frac{1}{p_k-2}\left(\frac{|\phi_t(x_k,t_k)|}{|D\phi(x_k,t_k)|}\right)^{p_k-4}\phi_t(x_k,t_k)^3\le  \frac{|D\phi(x_k, t_k)|^2\Delta\phi(x_k,t_k)}{p_k-2} +\Delta_\infty\phi(x_k,t_k).
\end{equation}
It follows that $-\Delta_{\infty}\phi(x_0,t_0)\le 0$ in the limit as $k\rightarrow \infty$. And as the right hand side of \eqref{PtendInf} is bounded, it must be that $\phi_t(x_k, t_k)\leq |D\phi(x_k, t_k)|$ for all $k$ large enough. Hence, $-|D\phi(x_0, t_0)|+\phi_t(x_0, t_0)\le 0$; in particular, \eqref{WantInfPara} holds.
\par Now assume $\phi_t(x_0,t_0)=0$. If in addition $|D\phi(x_0,t_0)|=0$, then clearly $-\Delta_\infty\phi(x_0,t_0)\le 0$. If $|D\phi(x_0,t_0)|\neq 0$, then $|D\phi(x_k,t_k)|\neq 0$ for all
$k$ large and \eqref{PtendInf} implies $-\Delta_\infty\phi(x_0,t_0)\le 0$ in the limit as $k\rightarrow \infty$. In either case, \eqref{WantInfPara} holds.

\par Finally, suppose $\phi_t(x_0,t_0)<0$. If additionally, $|D\phi(x_0,t_0)|+\phi_t(x_0,t_0)\le 0$, then clearly \eqref{WantInfPara} follows. Otherwise, 
$|D\phi(x_0,t_0)|+\phi_t(x_0,t_0)>0$ and in particular, $|D\phi(x_k,t_k)|+\phi_t(x_k,t_k)>0$ for all $k$ large. Passing to the limit in \eqref{PtendInf} gives $-\Delta_\infty\phi(x_0,t_0)\le 0$. 
In either case, again we have \eqref{WantInfPara}. 

\par It is now routine to verify that \eqref{WantInfPara} holds if $v-\phi$ only has a local maximum at $(x_0, t_0)$. Moreover, our proof that $v$ is a subsolution immediately extends to 
a proof that $v$ is a supersolution since $G_\infty$ is an odd function:
\begin{equation}\label{Godd}
G_\infty(-a,-\xi, -X)=-G_\infty(a,\xi, X). \nonumber
\end{equation}
for each $a\in \R$, $p\in \R^n$ and symmetric $n\times n$ matrix $X$.
\end{proof}
In \cite{JLM}, it was shown that $\lambda_\infty:=\lim_{p\rightarrow \infty}\lambda_p^{1/p}$ exists.  We conjecture that for any viscosity solution $v$ of \eqref{infParabolic}, the limit $\psi(x):=\lim_{t\rightarrow \infty}e^{\lambda_\infty t}v(x,t)$ exists uniformly in $x\in \Omega$ and is an 
infinity ground state. That is, $\psi$ is a viscosity solution of the PDE 
$$
\begin{cases}
G_\infty(-\lambda_\infty \psi, D\psi, D^2\psi)=0,\quad & x\in \Omega\\
\hspace{1.44in} \psi=0, \quad & x\in\partial\Omega
\end{cases}.
$$
In particular, if $\psi>0$ 
$$
\begin{cases}
\min\{-\Delta_\infty \psi, |D\psi| -\lambda_\infty \psi\}=0,\quad & x\in \Omega\\
\hspace{1.8in} \psi=0, \quad & x\in\partial\Omega
\end{cases}.
$$
If our intuition is correct, then it is appropriate to interpret the flow \eqref{infParabolic} as a natural parabolic equation associated with the infinity Laplacian.
\appendix


\end{document}